\documentclass[12pt,a4paper]{amsart}
\usepackage{amsmath, amsthm, amscd, amsfonts}

\newtheorem{theorem}{Theorem}[section]
\newtheorem{lemma}[theorem]{Lemma}

\newtheorem{corollary}[theorem]{Corollary}
\theoremstyle{definition}
\newtheorem{definition}[theorem]{Definition}

\newtheorem{remark}[theorem]{Remark}
\numberwithin{equation}{section}

\begin{document}

\title[Shellable and Cohen-Macaulay  complete t-partite graphs]{Shellable and Cohen-Macaulay complete t-partite graphs}
\author[ Seyyede Masoome Seyyedi, Farhad Rahmati, Mahdis Saeedi]{ Seyyede Masoome Seyyedi, Farhad Rahmati, Mahdis Saeedi}

\address{ Faculty of  Mathematics and Computer
Science, Amirkabir University of Technology, P. O. Box 15875-4413,
Tehran, Iran.} \email{mseyyedi@aut.ac.ir}

\address{ Faculty of  Mathematics and Computer
Science, Amirkabir University of Technology, P. O. Box 15875-4413,
Tehran, Iran.}

\email{frahmati@aut.ac.ir}

\address{ Faculty of  Mathematics and Computer
Science, Amirkabir University of Technology, P. O. Box 15875-4413,
Tehran, Iran.}

\email{mahdis.saeedi@aut.ac.ir}

\begin{abstract}
Let $G$ be a simple undirected graph. We find the number of
maximal independent sets in complete t-partite graphs. We will
show  that vertex decomposability and shellability are equivalent
in this graphs. Also, we obtain an equivalent condition for being
Cohen-Macaulay in complete t-partite graphs.
\end{abstract} \maketitle
\section{Introduction}
\noindent Let $G$ be a (simple, undirected, finite) graph. The
sets of vertices and edges of a graph $G$ are denoted by $V(G)$
and $E(G)$, respectively. A set $S\subseteq V(G)$ is independent
if no edge of $G$ has both its endpoints in $S$. An independent
set $S$ is maximal if no independent set of $G$ properly contain
$S$. Let m.i.s be the set of all maximal independent sets in $G$.
Let $i_{M}(G)$ denote the number of m.i.s in $G$.
\begin{definition}
We say that a graph is complete t-partite graph if its vertex set $V(G)$ can be partitioned into disjoint subsets $V_{1},\ldots,V_{t}$ such that $|V_i|=n_{i}$ for all $1\leq
i\leq t$ and $G$ contains every edge joining $V_i$ and $V_j$ for all $i\neq j$ and $1\leq i,j\leq t$.
\end{definition}
\begin{definition}
 We say that a graph $G$ is t-partite graph if its vertex set $V(G)$
 can be partitioned into disjoint subsets $V_1,\ldots,V_t$ such
that every edge of $G$ joins a vertex of $V_i$ to a vertex of
$V_j$ for $i\neq j$ and $1\leq i,j\leq t$.
 Moreover, in this paper we consider $t$ as the smallest number that has
above properties.
\end{definition}
Let $G$ be a simple undirected graph with the vertex set $V(G)
=\{x_1,\ldots,x_n\}$ and the edge set $E(G)$. The edge ideal of
$G$ is $I(G)$ as an ideal of
 the polynomial ring $R=k[x_1,\ldots,x_n]$ over a
field $k$ that generated by  all monomials $x_{i}x_{j}$ such
that $x_{i}$ is adjacent to $x_{j}$ in $G$. We can associate to G
the simplicial complex $\Delta_G$ that the faces of $\Delta_G$
are independent sets or stable sets of G and $I_{\Delta_G} =
I(G)$.
\begin{definition}
A simplicial complex $\Delta$ is called shellable if the facets
(maximal faces) of $\Delta$ can be ordered $F_1,\ldots,F_s$ such
that for all $1\leq i<j\leq s$, there exists some $v\in F_j\setminus
F_i$ and some $l\in \{1,\ldots,j-1\}$ with $F_j\setminus
F_l=\{v\}$.\\ We call $F_1,\ldots,F_s$ a shelling of $\Delta$ when
the facets have been ordered with respect to the definition of
shellable.
\end{definition}
A graph $G$ is called shellable, if the simplicial complex $\Delta
_G$ is a shellable simplicail complex.
\begin{definition}
A k-coloring of a graph $G$ is a labeling $f: V (G)\rightarrow S$
where $|S|= k$. The labels are colors; the
 vertices of one color
form a color class. A k-coloring is proper if adjacent vertices
have different labels. A graph is k-colorable if it has a proper
k-coloring. The chromatic number of graph G, $\chi(G)$, is the
least $k$ such that G is k-colorable.
\end{definition}
The problem finding $i_{M}(G)$, even lower or upper bounds  for it, is considerable lately. $i_{M}(G)$ is often known as the Fibonacci
number, or in mathematical chemistry as the Merrifield-Simmons index or the $\sigma$-index.
The study was initiated by Prodinger and Tichy in [9]. In the paper, first we present a lower bound for $i_{M}(G)$ where $G$ is a graph  with $\chi(G)$  as chromatic number and next we conclude that  $\chi(G)= i_{M}(G)=t$  in complete t-partite graphs.
 By using it, we obtain an equivalent condition  for shellability in complete t-partite graphs as following theorem:\\
\begin{theorem}\label{main}
Let $G$ be a complete t-partite graph. $G$ is shellable if and
only if  $G$  be t-colorable such that exactly one of color
classes has arbitrary elements and other classes have only one
element.
\end{theorem}
\begin{definition}
 A simplicial complex $\Delta$ is recursively defined to be vertex decomposable
if it is either a simplex, or else has some vertex $v$ so that\\
i)Both $\Delta\setminus v$ and $link_{\Delta} ^{v}$ are vertex
decomposable and  \\ ii)No face of $link_{\Delta}^{v}$ is a facet
of $\Delta\setminus v$, where \\$$link_{\Delta}^{F}=\{G: G\cap
F=\emptyset , G\cup F\in \Delta\}.$$
\end{definition}
We call a graph $G$  vertex decomposable if the simplicial complex
$\Delta_{G}$ is vertex decomposable.
\\By using above theorem, we present an equivalent condition for vertex decomposability in complete t-partite graphs
as following corollary:\\
\begin{corollary}\label{main}
Let $G$ be a complete t-partite graph. $G$ is vertex decomposable
if and only if $G$ be t-colorable  such that exactly one of color
classes has arbitrary elements and other classes have only one
element.
\end{corollary}

\begin{definition}
A local ring $(R, m)$ is called Cohen-Macaulay if $depth(R) = dim(R)$. If $R$ is non
local and $R _p$ is a Cohen-Macaulay local ring for all $p \in Spec(R)$, then we say that $R$ is a
Cohen-Macaulay ring.
\end{definition}

\begin{definition}
The graph $G$ is said to be Cohen-Macaulay over the field $k$ if $R/ I ( G )$ is a
Cohen-Macaulay ring.
\end{definition}
\begin{definition}
Let $G$ be a graph with vertices $V(G)=\{x_1,\ldots,x_n\}$. A
subset $A\subseteq V(G)$  is a vertex cover of $G$ if every edge in $G$ is incident to
some vertex in $A$. A vertex cover $A$ is minimal if no proper subset of
$A$ is a vertex cover.
\end{definition}

\begin{definition}
A graph $G$ is called unmixed if all the minimal vertex covers
of $G$ have the same number of elements.
\end{definition}
Herzog, Hibi, and Zheng [8] proved that a chordal graph
is Cohen-Macaulay if and only if it is unmixed. By using it, we show  when a complete t-partite graph is Cohen- Macaulay as following corollary:
\begin{corollary}\label{main}
Let $G$ be a complete t-partite graph. $G$ is Cohen-Macaulay graph if and only if $G$ be t-colorable such that all color classes have exactly one element.
\end{corollary}

\section{Main results}
\begin{lemma}\label{main}
Let $G$ be a graph. If $\chi(G)=t$ then $i_{M}(G)\geq t$.
\end{lemma}
\begin{proof}
Since  $\chi(G)=t$ then there exist $t$ color classes for $G$,
hence we consider $G$ as a t-partite graph. We suppose that
$V_{i}$ is set of elements i-th color class. By definition,
$V_{i}$ is  a  independent set of $G$, then there exists  maximal
independent set $F_i$ that $V_i\subseteq F_i$, for all $1\leq
i\leq t$. Hence, there exist at least $t$ maximal independent sets
for $G$.
\end{proof}
\begin{remark}\label{main}
By definition of complete t-partite graphs we have that
$i_{M}(G)= t$.
\end{remark}
The first author represented sufficient condition  of following
theorem in [11].
\begin{theorem}\label{main}
Let $G$ be a complete t-partite graph. $G$ is shellable if and
only if  $G$  be t-colorable such that exactly one of color
classes has arbitrary elements and other classes have only one
element.
\end{theorem}
\begin{proof}
$\Leftarrow)$ We have to find a shelling $F_1,\ldots,F_t$ for
$\Delta_G$. Let $V(G)=\{x_1,\ldots,x_n\}$. Since proper t-vertex
coloring gives a  partition of $V(G)$ into $t$ color classes, we
suppose that the set of elements of i-th color class be $V_i$. By
assumption,\\$V_1=\{x_1,\ldots,x_m\}$, $m=n-t+1$,
$V_i=\{x_{m+i-1}\}$ for all $2\leq i\leq t$. Any $V_i$, $1\leq
i\leq t$ is an independent set. Since for any $2\leq i\leq t$, if
$x_{m+i-1}\in V_1$ then we can replace $V_1$ by $V_{1}\cup V_{i}$
and obtain a  $t-1$-partition for $G$ that is a contradiction to
the assumption that $t$ is the smallest number with this
property. So $V_1$ is a maximal independent set of $G$ and a
facet of $\Delta_G$. We put $F_1=V_1$. On the other hand, we have
$V_i=\{x_{m+i-1}\}$ as an independent set of $G$, for all $1\leq i\leq t$. With the same
argument, we have $F_i= V_i$. Therefore, we find an ordering on
the facets of $\Delta_G$ as follows:\\$F_1=\{x_1,\ldots,x_m\}$,
$F_2=\{x_{m+1}\}$ ,$\ldots$, $F_t=\{x_{m+t-1}\}.$
\\Since $F_i\setminus F_1=\{x_{m+i-1}\}$ for
all $2\leq i\leq t$, then $F_1,\cdots F_t$ is a shelling of
$\Delta_G $.\\$\Rightarrow)$Let $V_1,\ldots,V_t$ be a partition of
$V(G)$ by definition complete t-partite graph.
 Since $i_{M}(G)= t$ then $\Delta _G$ has exactly $t$ facets. Let
$F_1,\ldots,F_t$ be a shelling of $\Delta_G$. Since $V_i$'s are
maximal independent sets then
$\{V_1,\ldots,V_t\}$=$\{F_1,\ldots,F_t\}$. Without loss of
generality, suppose that $F_i= V_i$ for all $2\leq i\leq t$. We
know that $\{F_1,\ldots,F_t\}$ is a shelling, hence there exists
$x_2\in F_2 \setminus F_1$ such that $F_2 \setminus F_1=\{x_2\}$.
Therefore $F_2=(F_2 \setminus F_1)\cup(F_2\cap
F_1)=\{x_2\}$.\\Now, suppose by induction that
$F_1=\{x_1,\ldots,x_m\}$, $F_2=\{x_{m+1}\}$,\ldots,$F_{i}=\{x_{m+i-1}\}$.
Since $\{F_1,\ldots,F_t\}$ is a shelling of $\Delta_G$ then there
exists $x_{i+1}\in F_{i+1} \setminus F_1$ and $l\in
\{1,\ldots,i\}$ such that $F_{i+1} \setminus F_l=\{x_{i+1}\}$,
then  $F_{i+1}=(F_{i+1} \setminus F_l)\cup(F_{i+1}\cap
F_l)=\{x_{i+1}\}$. Hence, one of color classes has arbitrary
elements and other classes have exactly one element.
\end{proof}
\begin{corollary}\label{main}
Let $G$ be a complete t-partite graph. $G$ is vertex decomposable
if and only if $G$ be t-colorable such that exactly one of color
classes has arbitrary elements and other classes have only one
element.
\end{corollary}
\begin{proof}
$\Rightarrow)$ We suppose that for any proper t-vertex coloring of
$G$, there exist at least two classes with at least two elements.
Then $G$ isn't shellable by previous theorem. Hence, $G$ isn't
vertex decomposable by [5,Corollary 7].
\\ $\Leftarrow)$
 By assumption, we
conclude that $G$ is a chordal graph. Hence, $G$ is vertex
decomposable by [5,Corollary 7].
\end{proof}
\begin{theorem}\label{main}
Let $G$ be a complete t-partite graph. $G$ is unmixed if and only if $G$ be t-colorable such that all color classes have the same number of elements.
\end{theorem}
\begin{proof}
Since any minimal vertex cover of
complete t-partite graph $G$ contains all the elements of $t-1$
classes, then  any $t-1$ classes have the same number of elements
if and only if cardinality of classes be similar.
\end{proof}
\begin{corollary}\label{main}
Let $G$ be a complete t-partite graph. $G$ is Cohen-Macaulay
graph if and only if $G$ be t-colorable such that all color
classes have exactly one element.
\end{corollary}
\begin{proof}
$\Leftarrow)$ By assumption, we have $G=K_t$ then $G$ is a chordal
graph. By[8,Theorem 2.1]$G$ is Cohen-Macaulay if and only if $G$
is unmixed. Hence, the conclusion follows from Theorem
2.5.\\$\Rightarrow)$ We suppose that one of color classes has at
least two elements.Therefore, $G$ isn't unmixed and hence $G$
isn't Cohen-Macaulay by[4,Proposition 6.1.21].
\end{proof}
By above Corollary and Theorem 2.3, we have following corollary:
\begin{corollary}\label{main}
Let $G$ be a complete t-partite graph. The property of being Cohen-Macaulay for
$G$ is equivalent to being shellabel if and only if $G$ is t-colorable such that all color
classes have exactly one element.
\end{corollary}

\end{document}